\documentclass{opt2025} 

\usepackage{algorithm,algorithmic}
\usepackage{bm}
\usepackage{multicol}

\newcommand{\alglinelabel}{%
  \addtocounter{ALC@line}{-1}
  \refstepcounter{ALC@line}
  \label
}

\newtheorem{assumption}{Assumption}

\title[Revisiting the Geometrically Decaying Step Size]{Revisiting the Geometrically Decaying Step Size: \\Linear Convergence for Smooth or Non-Smooth Functions}

\optauthor{%
\Name{Jihun Kim} \Email{jihun.kim@berkeley.edu}\\
\addr University of California, Berkeley, CA 94720, United States}

\begin{document}

\maketitle

\begin{abstract}%
We revisit the geometrically decaying step size given a positive inverse condition number, under which a locally Lipschitz function shows linear convergence. The positivity does not require the function to satisfy convexity, weak convexity, quasar convexity, or sharpness, but instead amounts to a property strictly weaker than the assumptions used in existing works (\textit{e.g.}, weak convexity + sharpness). We propose a clean and simple subgradient descent algorithm that requires minimal knowledge of problem constants, applicable to either smooth or non-smooth functions.
\end{abstract}

\section{Introduction}
The goal of optimization is to find a true minimizer $x^*$ in the set of minimizers $\mathcal{X}^*\subseteq \mathbb{R}^n$ of a function $f(x)$, whether smooth or non-smooth.
Goffin \cite{goffin1977on} and Shor \cite{shor1985minimization} independently studied the geometrically decaying step sizes in the subgradient descent algorithm, presenting various scenarios of the decaying rate, showing that linear convergence is achieved under convexity and sharpness-type assumptions. 
Polyak \cite{polyak1969min} also proved linear convergence under similar conditions; however, the Polyak step size requires the information of the minimum function value $f(x^*)$. 
Recent studies have proposed alternative methods to achieve linear convergence. 
For example, \cite{yang2018RSG} suggested averaging past consecutive iterates to obtain a robust estimate, while \cite{johnstone2019faster} proposed the descending-stairs step size scheme, which reduces the step size only occasionally. However, these approaches generally require complicated constants or involve redundant subroutines.

In this work, we revisit the classical geometrically decaying step size and propose a simple subgradient descent algorithm that achieves linear convergence for \textit{a locally Lipschitz—smooth or non-smooth—function} $f(x)$ under a general condition and requires minimal knowledge of constants. A similar line of work \cite{davis2018subgradient} established linear convergence under weak convexity and sharpness; however, these conditions are somewhat restrictive. In contrast, we  show that the ``positive inverse condition number'' stated in the next section turns out to be a property strictly weaker than the aforementioned weak convexity and sharpness. 
Meanwhile, the work \cite{hinder2020nearoptimal} examines quasar-convex functions—a generalization of star-convexity \cite{nesterov2006cubic} under which every line segment from a minimizer $x^*$ to any other point preserves convexity. We also establish that
the positive inverse condition number holds for all quasar-convex functions and extends to functions beyond this class.

Our focus is on this generalized \textit{positive inverse condition number} assumption, which holds not only for arbitrary norm functions but also for a class of nonconvex functions that may not satisfy sharpness, weak convexity, or quasar convexity. From a geometric perspective, a positive inverse condition number implies that the negative subgradient direction forms an acute angle with the direction toward a minimizer. This property can be interpreted in terms of cosine similarity, which is experimentally investigated in \cite{guille-escuret2024noWrongTurns}. Such a property frequently arises in applications involving  nonconvex functions without spurious local minima, including phase retrieval \cite{netrapalli2013phase, candes2015phase}, matrix recovery \cite{chen2014convergence,bi2022local}, image alignment \cite{zhang2023prise}, and dynamical systems \cite{hardt2018gradient}, even under adversarial environments \cite{kim2024prevailing}.  As one application, in dynamic programming (DP), it has been shown that if each DP subproblem has no spurious local minima, then one can obtain a global solution to the entire (one-shot) optimization problem \cite{BhandariRusso2024, kim2024landscape}, in which case the one-shot problem can be solved by repeatedly leveraging the property that will be discussed throughout the paper.

\textit{Notation: } We denote by $\|\cdot \|$ the $\ell_2$-norm of a vector, and by $\langle \cdot, \cdot \rangle$ the inner product of two vectors.  For a point $x$ and a set $X$,  $\text{dist}(x;X)$ denotes $\inf_{y\in X} \|x-y\|$. For two sets $X$ and $Y$,  $X\setminus Y$ denotes the set of elements in $X$  not in $Y$. 

\section{Assumption for Linear Convergence}

In this section, we begin by presenting the following assumption for the geometrically decaying step size scheme to achieve linear convergence. 
In this regard, we define the set of interest $S=\{x\in \mathbb{R}^n : \text{dist}(x; \mathcal{X}^*)\leq \text{dist}(x_0; \mathcal{X}^*)\}$, given an initial point $x_0\in\mathbb{R}^n$. 


\begin{assumption}[Positive Inverse Condition Number]\label{posconnum}
   Consider a locally Lipschitz function $f(x)$ over the set $S$. We define the inverse condition number $\bar{\mu}$ as
    \begin{align}\label{condnum}
            \bar{\mu} = \inf_{x^*\in \mathcal{X}^*, x\in S\setminus \mathcal{X}^*  }\inf_{u\in \partial^\circ f(x)\setminus \{0\}} \frac{\langle u, x-x^* \rangle}{\|u\|\|x-x^*\|},
    \end{align}
    where $\partial^\circ f(x)$ is a Clarke differential \cite{clarke1990optimization}, whose elements are called generalized subgradients.
    We assume that $\bm{\bar{\mu} > 0}$.
\end{assumption}
\begin{remark}
    Note that $\partial^\circ f(x)$ is nonempty, closed, and convex for all $x\in S$, since $f(x)$ is assumed to be locally Lipschitz. We also note that $\bar{\mu} \leq 1$ holds due to the Cauchy-Schwarz inequality.
\end{remark}

Under Assumption \ref{posconnum}, $0\in \partial^\circ f(x)$ is an equivalent condition to optimality, and thus serves as a stopping criterion for the generalized subgradient descent algorithm. 

\begin{lemma}\label{optimal}
    Suppose that $\langle u, x-x^*\rangle > 0$ holds for all $ x^*\in \mathcal{X}^*, x\in S\setminus \mathcal{X}^*$, and any $u \in \partial^\circ f(x)$. Then, a point $s\in S$ is a minimizer of $f$ if and only if $0\in \partial^\circ f(s)$. 
\end{lemma}
\begin{proof}
    If $s\in S$ is a minimizer, it is evident that $0\in \partial^\circ f(s)$ (see Proposition 2.3.2 in \cite{clarke1990optimization}). To prove the converse, suppose that $0\in \partial^\circ f(s)$ holds but $s$ is not a minimizer of $f$. Then, due to the precondition of the lemma, we arrive at $\langle 0, s -x^*\rangle >0$ for all $x^*\in \mathcal{X}^*$, which yields the contradiction. Thus, $s$ is a minimizer of $f$.
\end{proof}

\begin{remark}
 Lemma \ref{optimal} implies that if $x \in S \setminus \mathcal{X}^*$ (\textit{i.e.}, not a minimizer), then $0 \notin \partial^\circ f(x)$. Hence, in the definition of the inverse condition number in \eqref{condnum}, one can replace $u \in \partial^\circ f(x) \setminus \{0\}$ with $u \in \partial^\circ f(x)$, while keeping the quantity well-defined. 
\end{remark}

In the following two lemmas, we now provide some natural sufficient conditions for $\bar{\mu}>0$ to hold.

\begin{lemma}\label{weakly}
    Suppose $f(x)$ is locally Lipschitz over the set $S$ and $\rho$-weakly convex with respect to $x^*$; \textit{i.e.}, there exists $M>0,~ \rho \geq 0$ such that
\begin{align*}
    & \|u\|\leq M \quad \text{and} \quad f( x^* )\geq f(x) + \langle u, x^*-x\rangle - \frac{\rho}{2}\|x^* -x\|^2, \quad \forall u \in \partial^\circ f(x)
\end{align*}
   for all $x^*\in \mathcal{X}^*$ and $x\in S\setminus \mathcal{X}^*  $.  Suppose also that the function has sharpness; \textit{i.e.}, there exists  $m>0$ such that 
    $f(x)-f(x^*)\geq m \|x - x^*\|$
    for all $x^*\in \mathcal{X}^*$ and $x\in S\setminus \mathcal{X}^*  $.
    When $\text{dist}(x_0; \mathcal{X}^*)\leq \frac{m}{\rho}$, we have $\bar{\mu}\geq \frac{m}{2M}>0. $
\end{lemma}

\begin{proof}
Using the definition of the inverse condition number, we have
\begin{align*}
\bar{\mu} \geq \frac{f(x)-f(x^*)-\frac{\rho}{2}\|x^*-x\|^2}{\|u\|\|x-x^*\|}\geq \frac{1}{M}\left(\frac{f(x)-f(x^*)}{\|x-x^*\|} - \frac{\rho}{2}\|x^*-x\|\right) \geq \frac{1}{M}(m-\frac{\rho}{2}\cdot \frac{m}{\rho})=\frac{m}{2M}
\end{align*}
for all $x^*\in \mathcal{X}^*, x\in S\setminus \mathcal{X}^*  $.
This completes the proof.
\end{proof}

\begin{lemma}\label{quasar}
     Suppose $f(x)$ is locally Lipschitz over the set $S$ and $\gamma$-quasar convex with respect to $x^*$; \textit{i.e.}, there exists $M>0$, $\gamma \in (0,1]$ such that 
     \begin{align*}
     & \|u\|\leq M \quad \text{and} \quad
     f(x^*) \geq f(x) + \frac{1}{\gamma} \langle u, x^*-x\rangle, \quad \forall u \in \partial^\circ f(x)
     \end{align*}
    for all $x^*\in \mathcal{X}^*$ and $x\in S\setminus \mathcal{X}^*  $. Suppose also that the function has sharpness; \textit{i.e.}, there exists  $m>0$ such that 
    $f(x)-f(x^*)\geq m \|x - x^*\|$
    for all $x^*\in \mathcal{X}^*$ and $x\in S\setminus \mathcal{X}^*  $. Then, we have
     $\bar{\mu} \geq \frac{\gamma m}{M}>0.$
\end{lemma}
\begin{proof}
    One can observe that 
     $   \bar{\mu} \geq\ \frac{\gamma(f(x)-f(x^*))}{\|u\|\|x-x^*\|}\geq \frac{\gamma m}{M}$
 for all $x^*\in \mathcal{X}^*$ and $x\in S\setminus \mathcal{X}^*  $.
\end{proof}

Lemmas \ref{weakly} and \ref{quasar}  respectively provide sufficient conditions for Assumption \ref{posconnum} to hold, but it requires the function to  have sharpness and be weakly-convex or quasar-convex.
To illustrate that our Assumption \ref{posconnum} can be regarded as a more general notion, we present examples where it holds even though the function lacks convexity, weak convexity, quasar convexity, or sharpness. For a pictorial illustration, we present graphs of these examples in Figure \ref{fig:examples}.

\begin{example}[Assumption \ref{posconnum} does not imply convexity]\label{convexisstrong}
Consider a function $f(x) = |3x| + \sin(|x|)$, which attains its unique minimum at $0$. Although the function is nonconvex,  we have $\bar{\mu} >0$ since
    \begin{align*}
        \langle f'(x), x\rangle =
            |x|\cdot (3+\cos (x)) , \quad x\neq 0,
    \end{align*}
    which is exactly $|x| |f'(x)|$. This implies that the inverse condition number is 1.
\end{example}

\begin{figure}[t]
  \centering
  \begin{minipage}[b]{0.24\textwidth}
    \centering
    \includegraphics[width=\linewidth, height=30mm]{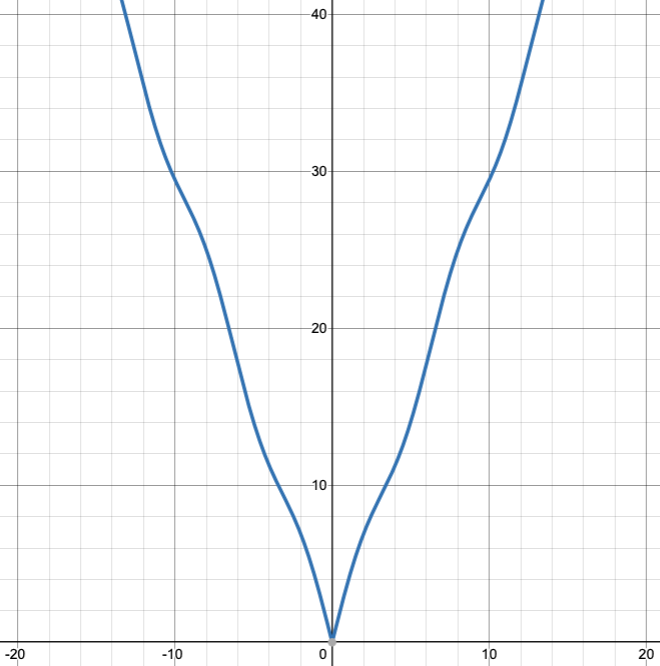}
    \text{(a) Example \ref{convexisstrong}} 
    \label{fig:a}
  \end{minipage}\hfill
  \begin{minipage}[b]{0.24\textwidth}
    \centering
    \includegraphics[width=\linewidth, height=30mm]{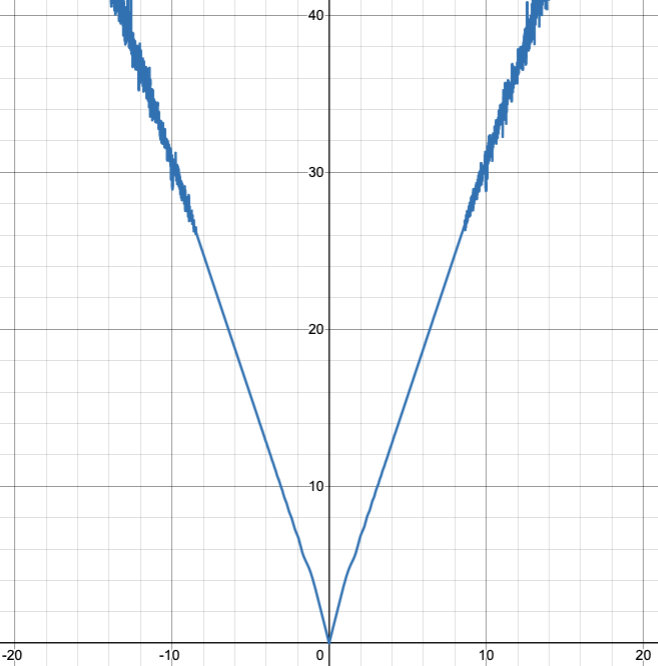}
    \text{(b) Example \ref{weakisstrong}}
    \label{fig:b}
  \end{minipage}
  \begin{minipage}[b]{0.28\textwidth}
    \centering
    \includegraphics[width=\linewidth, height=30mm]{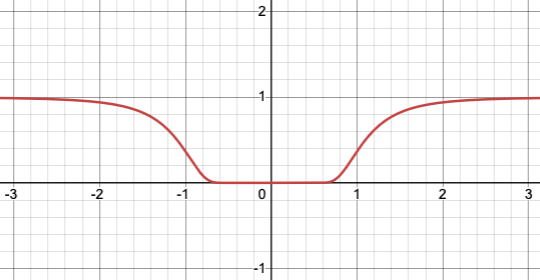}
    \text{(c) Example \ref{quasarisstrong}}
    \label{fig:c}
  \end{minipage}
  \hfill
  \begin{minipage}[b]{0.21\textwidth}
    \centering
    \includegraphics[width=\linewidth, height=30mm]{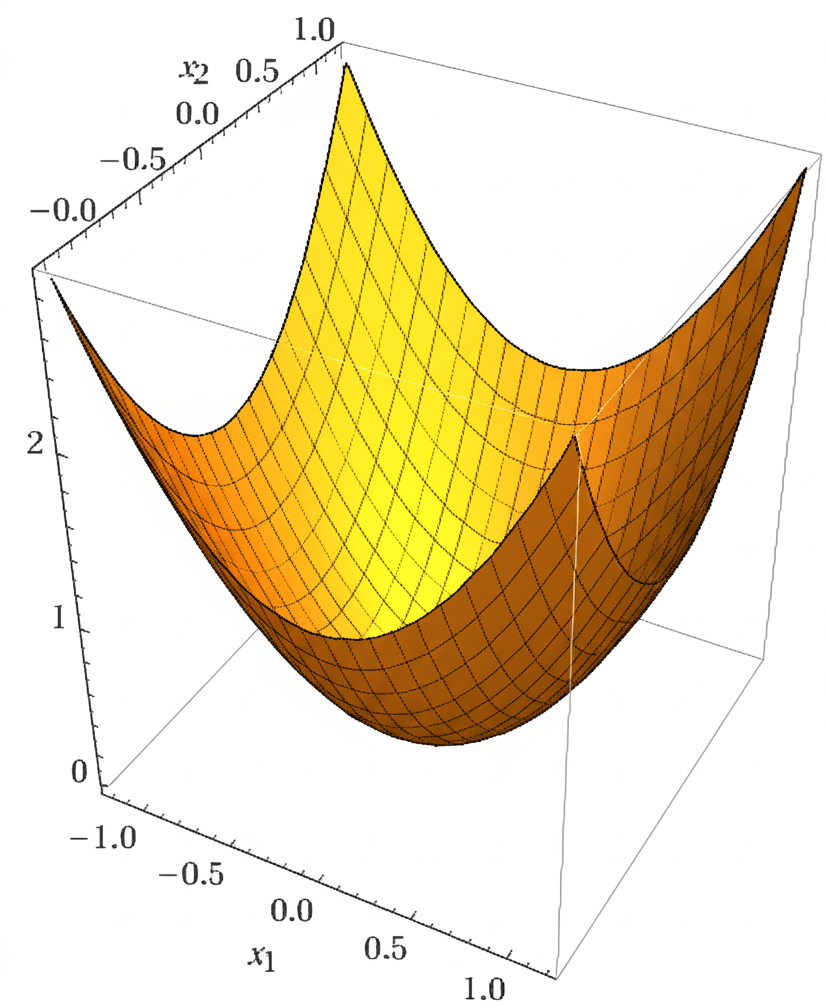}
    \text{(d) Example \ref{ex4}}
    \label{fig:d}
  \end{minipage}
  \caption{Examples of functions with positive inverse condition numbers but without (a) convexity, \\(b) weak convexity, (c) quasar convexity, or (d) sharpness}
  \label{fig:examples}
  \vspace{-2mm}
\end{figure}

\begin{example}[Assumption \ref{posconnum} does not imply weak convexity]\label{weakisstrong}
    Consider a function $f(x) = \int_0^{|x|} (3+\cos (s^3))ds$, which attains its minimum at $0$. Since the function is twice-differentiable, the weak convexity is equivalent to finding $\rho\geq 0$ such that $f''(x) \geq -\rho$ for all $x$. For $x>0$, we have
    \begin{align*}
        f''(x) = \frac{d}{dx}(3+\cos(x^3)) = -3x^2 \sin(x^3),
    \end{align*}
    which is unbounded below. However, we have $\bar{\mu} >0$ since
    \begin{align*}
         \langle f'(x) , x\rangle = 
           |x|\cdot (3+\cos(x^3)), \quad x\neq 0,     
    \end{align*}
    which is exactly $|x| |f'(x)|$. This implies that the inverse condition number is 1.
\end{example}

\begin{example}[Assumption \ref{posconnum} does not imply quasar convexity]\label{quasarisstrong}
Consider a function $f(x) = e^{-1/x^4}$ for $x\neq 0$, and $f(0)=0$, which attains its minimum at $0$. It fails to satisfy quasar convexity since there does not exist $\gamma >0$ such that 
\begin{align*}
    \gamma  \leq \frac{\langle f'(x), x\rangle}{f(x)} = \frac{\frac{4}{x^5}e^{-1/x^4} \cdot x}{e^{-1/x^4}} = \frac{4}{x^4}, \quad \forall x\neq 0,
\end{align*}
since the right term converges to $0$ as $x\to \pm\infty$. However, we have $\bar{\mu} >0$ since 
$    \langle f'(x) , x\rangle = \frac{4}{x^4}e^{-1/x^4} =|x| |f'(x)|,$
implying that the inverse condition number is 1.
\end{example}

\begin{example}[Assumption \ref{posconnum} does not imply sharpness]\label{ex4}
    Consider a function $f(x_1,x_2) = x_1^2 + x_2^2$, which attains its unique minimum at $0$. This function does not have sharpness in the sense that for any $m>0$, one can always find  $(x_1, x_2)$ such that 
    \[
    f(x_1,x_2) = x_1^2 + x_2^2 < m \sqrt{x_1^2 + x_2^2},
    \]
    especially when $0<x_1^2+x_2^2 < m^2$. However, we have 
$ \langle \nabla f(x_1,x_2), (x_1,x_2)\rangle = 2\sqrt{x_1^2+x_2^2} \cdot \sqrt{x_1^2+x_2^2}$,
    which implies that the inverse condition number is 1. 
\end{example}

For non-smooth functions, existing works \cite{davis2018subgradient, hinder2020nearoptimal} relied on weak convexity or quasar convexity, together with sharpness. On the other hand, for smooth functions, the quadratic growth condition $f(x)-f(x^*)\geq m \|x- x^*\|^2$ derived from strong convexity is a standard assumption to ensure linear convergence \cite{boyd2004convex, nesterov2004intro}. However,  $x_1^2+x_2^2$ has quadratic growth with $m=1$, yet not sharpness (see Example \ref{ex4}), indicating that sharpness is an overly restrictive requirement applicable only to functions that are non-smooth at the minimum. In the next section, we will show that a \textit{positive inverse condition number} suffices to guarantee linear convergence for either smooth or non-smooth functions.

\section{Geometrically Decaying Step Size Design}

In this section, we provide the generalized subgradient descent algorithm and provide our convergence analysis under the \textit{positive inverse condition number}.
We note that \cite{davis2018subgradient} and \cite{hinder2020nearoptimal} introduced a linear convergent algorithm only under the assumptions of Lemma \ref{weakly} and Lemma \ref{quasar}, respectively, whereas our analysis shows that linear convergence still holds under a strictly weaker condition.

\begin{algorithm}[!t]
  \caption{Generalized Subgradient Descent Algorithm}
  
  \begin{flushleft}
        \textbf{Input:} Inverse Condition number $\bar{\mu}$. Initial point $x_0$. 
    
  \end{flushleft}
  \begin{multicols}{2}
  \begin{algorithmic}[1]
  \STATE  Let $0<r \leq \min\{\bar{\mu}, \frac{1}{\sqrt{2}}\}$. 
    \FOR{time $t=0,1,\dots$}
\IF{$0\in \partial^\circ f(x_t)$}
\STATE Break. $x_t$ is a minimizer.

    \ELSE
        \STATE Pick $g_t\in \partial^\circ f(x_t)$.
        
        \STATE Set the step size $\eta_t = \frac{  r\left(1-r^2\right)^{t/2} \|x_0-x^*\|}{\|g_t\|}.$\alglinelabel{stepsize}
        \STATE Let $x_{t+1} =x_t - \eta_t g_t.$
        \ENDIF
    \ENDFOR 
  \end{algorithmic}
  \end{multicols}
  \label{Algorithm 1}
\end{algorithm}

\begin{theorem}\label{theorem1}
  Consider a minimizer $x^* \in \mathcal{X}^*$. Under Assumption \ref{posconnum}, Algorithm \ref{Algorithm 1} achieves
  \begin{align}
      \|x_t-x^*\|\leq \left(1-r^2\right)^{t/2} \|x_0-x^*\|. 
  \end{align}
\end{theorem}

\begin{proof}
    We prove this by induction. The base case $t=0$ is trivial. For the induction step, suppose that 
    \[
    \|x_s-x^*\|\leq \left(1-r^2\right)^{s/2} \|x_0-x^*\|
    \]
    holds. Let $\alpha$ be the constant that satisfies 
    \begin{align}\label{alphalinear}
         \alpha \|x_s-x^*\| = \left(1-r^2\right)^{s/2} \|x_0-x^*\|.
    \end{align}
    By the induction hypothesis, we clearly have $\alpha\geq 1$. Then, from the definition of the step size, we arrive at 
    \begin{align}\label{alphamM}
        \eta_s \|g_s\| =r \alpha  \|x_s-x^*\|. 
    \end{align}
    
    Then, we have 
    \begin{align*}
        \|x_{s+1} -  x^*\|^2 & = \|x_s- \eta_s g_s-x^*\|^2  = \|x_s-x^*\|^2 -2\eta_s \langle g_s , x_s-x^*\rangle + \eta_s^2 \|g_s\|^2\\&\underset{\text{(a)}}{\leq} \|x_s-x^*\|^2 -2\eta_s  \|g_s\|\|x_s-x^*\| \cdot r + \eta_s^2 \|g_s\|^2 \\  &\underset{\text{(b)}}{=} \|x_s-x^*\|^2 - 2r\alpha \|x_s-x^*\|  \cdot  \|x_s-x^*\| r + (r \alpha  \|x_s-x^*\|)^2 \\&= \left(1-2r^2 \alpha +r^2 \alpha^2 \right) \|x_s-x^*\|^2 \\ &\underset{\text{(c)}}{\leq} (  1-r^2 ) \alpha^2 \|x_s-x^*\|^2 \\&\underset{\text{(d)}}{=} (  1-r^2 )^{s+1} \|x_0-x^*\|^2  
    \end{align*}
    where (a) is due to  the definition of the inverse condition number and $ \bar{\mu}\geq r$, (b) is due to \eqref{alphamM}, and (d) is from \eqref{alphalinear}. The inequality (c) is derived from the fact that $r\leq \frac{1}{\sqrt{2}}$ and $\alpha \geq 1$ yields
    \begin{align*}
        &(  1-2r^2 )\alpha^2  + 2r^2\alpha  - 1 = (\alpha-1) \left((  1-2r^2 )\alpha  + 1\right)\geq 0.
    \end{align*}
   Thus, we achieve
   $ \|x_{s+1} - x^*\|\leq (  1-r^2 )^{\frac{s+1}{2}} \|x_0-x^*\|,$
    which completes the proof.
\end{proof}

\begin{algorithm}[!t]
  \caption{Generalized Subgradient Descent Algorithm without the exact estimation of $\|x_0-x^*\|$}
  
  \begin{flushleft}
        \textbf{Input:} Inverse Condition number $\bar{\mu}$. Initial point $x_0$. A constant $R$ that satisfies 
        \begin{align}\label{Rrange}
        \beta \|x_0-x^*\| \leq R \leq (1-\beta) \|x_0-x^*\|
        \end{align}
        for a pre-determined $0<\beta \leq 0.5$. 
    
  \end{flushleft}
\begin{multicols}{2}
     \begin{algorithmic}[1]
  \STATE  Let $0<r \leq \bar{\mu}$. 
    
    \FOR{time $t=0,1,\dots$}
    \IF{$0\in \partial^\circ f(x_t)$}
\STATE Break. $x_t$ is a minimizer.

    \ELSE
        \STATE Pick $g_t\in \partial^\circ f(x_t)$.
        
        \STATE Set the step size $\eta_t = \frac{  r\left(1+(  \beta^2-2\beta)r^2\right)^{t/2} R}{\|g_t\|}$.\alglinelabel{stepsize2}
        \STATE Let $x_{t+1} =x_t - \eta_t g_t.$
        \ENDIF
    \ENDFOR 
  \end{algorithmic}
\end{multicols}
  \label{Algorithm 2}
\end{algorithm}

  In Algorithm \ref{Algorithm 1},  one can choose $r$ from 
   the range $(0, \min\{\bar{\mu}, \frac{1}{\sqrt{2}}\}]$. A natural choice to maximize the linear convergence rate $\sqrt{1-r^2}$ would be to replace $\bar{\mu}$ with $\frac{m}{2M}$ or $\frac{\gamma m}{M}$, given that we can estimate $m$ and $M$ (see Lemmas \ref{weakly} and \ref{quasar}).
   Then, the only term that we are unaware of in the step size $\eta_t$ is  $\|x_0-x^*\|$; \textit{i.e.} the distance from the initial point to any true minimizer $x^*\in \mathcal{X}^*$. In the next theorem, we propose that the unknown $\|x_0-x^*\|$ need not be exactly estimated, but can instead be replaced by a sufficiently small constant and still achieve linear convergence.

\begin{theorem}\label{theorem2}
    Consider a minimizer $x^* \in \mathcal{X}^*$.  Under Assumption \ref{posconnum}, Algorithm \ref{Algorithm 2} achieves 
    \begin{align}
          \|x_t-x^*\|\leq \left(1+(\beta^2-2\beta)r^2\right)^{t/2} \|x_0-x^*\|.
    \end{align}
\end{theorem}

\begin{proof}
    We use a similar induction technique. The proof details can be found in Appendix \ref{alg2proof}.
\end{proof}

\begin{remark}
    Since $0<\beta \leq 0.5$ in Algorithm \ref{Algorithm 2}, we have $(\beta^2-2\beta)r^2 < 0$, which again ensures linear convergence of the algorithm. The additional constants required to determine the step size are $\beta$ and $R$. If $\beta$ is chosen sufficiently small,  one can increase confidence to obtain a correspondingly small value of $R$ that satisfies \eqref{Rrange}. However, one should account for a natural trade-off: setting $\beta$ too small may lead to a worse linear convergence ratio. 
\end{remark}

\section{Conclusion}
In this work, we  propose a geometically decaying step size scheme that achieves linear convergence to a minimizer under the assumption of a positive inverse condition number. We identify that existing assumptions—such as weak convexity or quasar-convexity, combined with sharpness—are common in applications yet impose strictly stronger requirements than a positive inverse condition number.
We first develop a generalized subgradient descent algorithm that requires knowledge of the distance from the initial point to a minimizer, and subsequently show that it can be replaced by a user-defined constant, potentially altering the convergence rate while preserving linear convergence. Our results provide a foundation for a more general assumption and convergence analysis to incorporate the possibility of extending the framework to a broader class of functions in a wider range of optimization areas, including \textit{stochastic, robust, nonconvex, and distributed} optimization.

\bibliography{bibref}

\clearpage
\appendix

\section{Proof for Theorem \ref{theorem2}}\label{alg2proof}
 We leverage a similar induction technique from Theorem \ref{theorem1}. The base case $t=0$ is trivial. For the induction step, suppose that 
     \[
    \|x_s-x^*\|\leq \left(1+(\beta^2-2\beta)r^2\right)^{s/2} \|x_0-x^*\|
    \]
    holds. Let $\alpha$ be a constant satisfying
    \begin{align}\label{alpha2}
    \alpha \|x_s-x^*\|= \left(1+(\beta^2-2\beta)r^2\right)^{s/2} \|x_0-x^*\|,
    \end{align}
where $\alpha\geq 1$ naturally holds. Let $q := \frac{R}{\|x_0-x^*\|}$.  Then, we have 
\begin{align}\label{ralpha2}
    \eta_s \|g_s\| = r\alpha \|x_s-x^*\| \frac{R}{\|x_0-x^*\|}=rq\alpha \|x_s-x^*\| .
\end{align}
Similar to the proof in Theorem \ref{theorem1}, we now arrive at 
 \begin{align*}
        \|x_{s+1} - x^*\|^2 & \leq \|x_s-x^*\|^2 -2\eta_s  \|g_s\|\|x_s-x^*\|\cdot r + \eta_s^2 \|g_s\|^2 \\  &\underset{\text{(a)}}{=} \|x_s-x^*\|^2 - 2rq\alpha \|x_s-x^*\|  \cdot  \|x_s-x^*\| r + (rq \alpha  \|x_s-x^*\| )^2 \\&= \left(1-2r^2 q\alpha  +r^2  q^2 \alpha^2\right) \|x_s-x^*\|^2 \\ &\underset{\text{(b)}}{\leq} (  1+(\beta^2-2\beta)r^2 ) \alpha^2 \|x_s-x^*\|^2 \\&\underset{\text{(c)}}{=} (  1+(\beta^2-2\beta)r^2 )^{s+1} \|x_0-x^*\|^2 ,
    \end{align*}
    where (a) and (c) comes from \eqref{ralpha2} and \eqref{alpha2}, respectively. 
For (b), observe that $\beta \leq q\leq 1-\beta$. It follows that 
\begin{subequations}\label{betaq}
    \begin{align}
  & \beta^2 - 2\beta -q^2 \geq \beta^2 -2\beta -(1-\beta)^2  = -1,
  \\&\beta^2 - 2\beta  - q(q-2)\geq \beta^2 -2\beta - \beta(\beta-2) = 0,
\end{align}
\end{subequations}
which leads to
\begin{align*}
   \bigr[ 1+ \left(\beta^2 -2\beta -q^2\right)r^2\bigr]&\alpha^2 +2r^2 q \alpha -1 \\&= (\alpha-1)\left(\left[ 1+ \left(\beta^2 -2\beta -q^2\right)r^2\right]\alpha+1\right) + (\beta^2-2\beta-q^2+2q) r^2 \alpha\\&\geq (\alpha-1) ((1-r^2)\alpha + 1) +0\geq 0,
\end{align*}
where the first inequality follows from $\alpha\geq 1$ and \eqref{betaq}, and the second from $\alpha\geq 1$ and $r=\bar{\mu}\leq 1$.
    Thus, the induction step is complete.

\end{document}